\documentclass[12pt,fleqn,leqno]{amsart}
\usepackage{amsfonts,amssymb}
\usepackage{mathrsfs}
\usepackage[bb=boondox]{mathalfa} 
\usepackage{enumitem}

\usepackage[svgnames]{xcolor}
\usepackage[colorlinks,linkcolor=blue,citecolor=Green]{hyperref}
\usepackage{titlesec}

\titleformat{\section}[block]
 {\bfseries}
 {\thesection.}
 {\fontdimen2\font}
 {}

\newtheorem{theorem}{Theorem}[section]
\newtheorem{lemma}[theorem]{Lemma}
\newtheorem{corollary}[theorem]{Corollary}
\newtheorem{proposition}[theorem]{Proposition}

\DeclareMathOperator{\R}{\mathbb{R}}
\DeclareMathOperator\sto{\leadsto}
\DeclareMathOperator{\uhr}{\upharpoonright}  
\DeclareMathOperator{\sel}{\mathnormal{se}\mkern-1.5mu\ell}
\DeclareMathOperator{\cws}{cws}
\DeclareMathOperator{\ws}{ws}
\DeclareMathOperator{\diam}{diam}
\DeclareMathOperator{\node}{node}

\renewcommand{\emptyset}{\varnothing}

\setlist{noitemsep}
\setenumerate{labelindent=\parindent,label=\upshape{(\alph*)}}

\numberwithin{equation}{section} 
\begin{document}

\author{Valentin Gutev}

\address{Department of Mathematics, Faculty of Science, University of
   Malta, Msida MSD 2080, Malta}
\email{valentin.gutev@um.edu.mt}



\title{Weak Selections and Suborderable Metrizable Spaces}

\begin{abstract}
  Each continuous weak selection for a space $X$ defines a coarser
  topology on $X$, called a selection topology. Spaces whose topology
  is determined by a collection of such selection topologies are
  called continuous weak selection spaces. For such spaces,
  Garc\'{\i}a-Ferreira, Miyazaki, Nogura and Tomita considered the
  minimal number $\cws(X)$ of selection topologies which generate the
  original topology of $X$, and called it the cws-number of $X$. In
  this paper, we show that $\cws(X)\leq 2$ for every semi-orderable
  space $X$, and that $\cws(X)=2$ precisely when such a space $X$ has
  two components and is not orderable. Complementary to this result,
  we also show that $\cws(X)=1$ for each suborderable metrizable space
  $X$ which has at least 3 components.
\end{abstract}

\dedicatory{Dedicated to Professor Salvador Garc\'{\i}a-Ferreira on
  the occasion of his 60th birthday}

\date{\today}
\maketitle

\section{Introduction}

All spaces in this paper are assumed to be Hausdorff. For a space $X$,
let
\[
  \mathscr{F}_2(X)=\{S\subset X: 1\leq |S| \leq 2\}.
\]
A map $\sigma:\mathscr{F}_2(X)\to X$ is a \emph{weak selection} for
$X$ if $\sigma(S)\in S$ for every $S\in \mathscr{F}_2(X)$. Every weak
selection $\sigma$ generates an order-like relation $\leq_\sigma$ on
$X$ \cite[Definition 7.1]{michael:51} defined by $x\leq_\sigma y$ if
$\sigma(\{x,y\})=x$, and we often write $x<_\sigma y$ to express that
$x\leq_\sigma y$ and $x\neq y$. The relation $\leq_\sigma$ is very
similar to a linear order on $X$, but is not necessarily transitive. A
weak selection $\sigma$ for $X$ is \emph{continuous} if it is
continuous with respect to the Vietoris topology on
$\mathscr{F}_2(X)$, which can be expressed by the property that for
every $x,y\in X$ with $x<_\sigma y$, there are open sets
$U,V\subset X$ such that $x\in U$, $y\in V$ and $s<_\sigma t$ for
every $s\in U$ and $t\in V$, see \cite[Theorem
3.1]{gutev-nogura:01a}. Continuity of a weak selection $\sigma$
implies that all $\leq_\sigma$-open intervals
$(\leftarrow, x)_{\leq_\sigma}=\{y\in X: y<_\sigma x\}$ and
$(x,\to)_{\leq_\sigma}= \{y\in X: x<_\sigma y\}$, $x\in X$, are open
in $X$ \cite{michael:51}, but the converse is not necessarily true
\cite[Example 3.6]{gutev-nogura:01a} (see also \cite[Corollary 4.2 and
Example 4.3]{gutev-nogura:09a}). For an extended review of (weak)
hyperspace selections, the interested reader is refereed to
\cite{gutev-2013springer}. \medskip

If $\sigma$ is a continuous weak selection for $X$, then it remains
continuous with respect to any other topology on $X$ which is finer
than the original one \cite[Corollary 3.2]{gutev-nogura:01a}. Looking
for a possible coarsest topology with this property, a natural
topology $\mathscr{T}_{\sigma}$ on $X$ was associated to $\sigma$ in
\cite{gutev-nogura:01a}. It was called a \emph{selection topology},
and was defined following exactly the pattern of the usual open
interval topology utilising the collection of $\leq_\sigma$-open
intervals
$\mathscr{S}_{\sigma}=\big\{(\gets,x)_{\leq_\sigma},
(x,\to)_{\leq_\sigma}: x\in X\big\}$ as a subbase. It was shown in
\cite{gutev-nogura:03b} that $\mathscr{T}_{\sigma}$ is regular, and in
\cite{hrusak-martinez:09b} that $\mathscr{T}_{\sigma}$ is also
Tychonoff. Some pathological examples of continuous weak selections
that are not continuous with respect to the selection topology they
generate were given in
\cite{artico-marconi-pelant-rotter-tkachenko:02,gutev-nogura:01a} (see
also \cite{gutev-nogura:03b,gutev-nogura:09a}).  Subsequently,
answering a question of \cite{gutev-nogura:06b}, it was shown in
\cite{hrusak-martinez:09b} that if there is a coarsest topology on a
given set so that a weak selection defined on it is continuous, then
this topology must be precisely the selection topology determined by
the weak selection itself.\medskip

Regarding the distinction between the original topology and a
selection topology, the following spaces $(X,\mathscr{T})$ were
studied in \cite{hrusak-martinez:09b}: \emph{weakly determined by
  selections} if $X$ admits a weak selection $\sigma$ with
$\mathscr{T}=\mathscr{T}_\sigma$; \emph{determined by selections} if
$X$ admits a continuous weak selection $\sigma$ with
$\mathscr{T}=\mathscr{T}_\sigma$; and \emph{strongly determined by
  selections} if $X$ admits a continuous weak selection and
$\mathscr{T}=\mathscr{T}_\sigma$ for every continuous weak selection
$\sigma$ for $X$. Every orderable space is determined by selections,
and it was shown in \cite[Example 3.4]{hrusak-martinez:09b} that so
also is the Sorgenfrey line, which is suborderable but not
orderable. However, there are suborderable spaces which are not
determined by selections, for instance such a space is the subspace
\begin{equation}
  \label{eq:sel-top-30:1}
  X=(0,1)\cup\{2\}\subset \R.
\end{equation}
Every connected locally connected space which admits a continuous weak
selection is strongly determined by selections (see
\cite{nogura-shakhmatov:97a}); and every compact space that admits a
continuous weak selection is also strongly determined by
selections. Answering a question of \cite{gutev-nogura:06b}, it was
shown in \cite[Example 3.8]{hrusak-martinez:09b} that there is a space
which is strongly determined by selections and yet it is neither
(locally) compact nor (locally) connected. \medskip

The idea of spaces determined by selections was generalised in
\cite{garcia-nogura-2011a}. For a set $X$ and a family
$\{\mathscr{T}_\alpha: \alpha\in \mathscr{A}\}$ of topologies on $X$,
the \emph{supremum topology}
$\bigvee_{\alpha\in \mathscr{A}}\mathscr{T}_\alpha$ is the smallest
topology on $X$ which contains all topologies $\mathscr{T}_\alpha$,
$\alpha\in \mathscr{A}$. A topological space $(X,\mathscr{T})$ is a
(\emph{continuous}) \emph{weak selection space}
\cite{garcia-nogura-2011a} if
$\mathscr{T}=\bigvee_{\sigma\in \Sigma}\mathscr{T}_\sigma$ for some
collection $\Sigma$ of (continuous) weak selections for $X$. Some
basic properties of these spaces, also several examples, were provided
in \cite{garcia-nogura-2011a,MR3430989}.\medskip

For a continuous weak selection space $(X,\mathscr{T})$, the
\emph{cws-number} of $X$, denoted by $\cws(X)$,
\cite{garcia-nogura-2011a} is the minimal cardinality of a collection
$\Sigma$ of continuous weak selections for $X$ with
$\mathscr{T}=\bigvee_{\sigma\in \Sigma}\mathscr{T}_\sigma$. Similarly,
for a weak selection space $(X,\mathscr{T})$, the \emph{ws-number} of
$X$ \cite{garcia-nogura-2011a} is the minimal cardinality $\ws(X)$ of
a collection $\Sigma$ of weak selections for $X$ with
$\mathscr{T}=\bigvee_{\sigma\in \Sigma}\mathscr{T}_\sigma$. In these
terms, a space $X$ is weakly determined by selections iff $\ws(X)=1$,
and $X$ is determined by selections iff $\cws(X)=1$. Thus, every
orderable space $X$ satisfies $\cws(X)=1$, but the converse is not
necessarily true. The Sorgenfrey line $S$ is an example of a
non-orderable, suborderable space with $\cws(S) = 1$
\cite{hrusak-martinez:09b}. Readers who are more familiar with the
Michael line $M$ can use it as another example of a non-orderable
suborderable space with $\cws(M)=1$ \cite{garcia-nogura-2011a}. In
contrast, the space $X$ in \eqref{eq:sel-top-30:1} is a suborderable
space with $\cws(X)=2$ \cite{garcia-nogura-2011a}. In this regard, it
was shown in \cite{shimpei:2012} that $\cws(X)\leq 2$ for every
subspace $X\subset \R$. In fact, it was shown in \cite{shimpei:2012}
that $\cws(X)=2$ if and only if $X$ has exactly two connected
components one of which is compact while the other is an open interval
of $\R$, compare with \eqref{eq:sel-top-30:1}. \medskip

In this paper, we aim to extend the results of \cite{shimpei:2012} in
two directions, and give some natural explanation of the behaviour of
the cws-number in these cases. A space $X$ is \emph{semi-orderable}
\cite{gutev:07b} if it has a clopen partition into two orderable
spaces or, equivalently, if it is a topological sum of two orderable
spaces. Every semi-orderable space is suborderable, while the
Sorgenfrey line and the Michael line are suborderable but not
semi-orderable \cite[Example 4.12]{gutev:07b}. One of the simplest
examples of a semi-orderable space which is not orderable is the space
$X$ in \eqref{eq:sel-top-30:1}.  The importance of semi-orderable
spaces was justified by the fact that a space $X$ is semi-orderable if
and only if it is a topological sum of orderable spaces \cite[Theorem
4.2]{gutev:07b}. In the next section, we show that $\cws(X)\leq 2$ for
every semi-orderable space $X$ (Theorem
\ref{theorem-sel-top-25:1}). Furthermore, we show that, in this case,
$\cws(X)=2$ precisely when $X$ has two components and is not
orderable. This result is based on special ``discretely circular''
weak selections which transform the elements of a partition of a set
$X$ into a discrete partition with respect to the selection topology
they generate, see Lemma \ref{lemma-sel-top-4:1}. In the rest of the
paper, we extend this construction of weak selections from a single
partition to a system of partitions, see Theorem
\ref{theorem-sel-top-23:1}. Based on this, we show that $\cws(X)=1$
for each suborderable metrizable space $X$ which has at least 3
components, see Theorem \ref{theorem-sel-top-19:1}. In fact, if such a
space has finitely many components or is locally connected, then it is
a topological sum of its components, hence it is semi-orderable. Thus,
the essential case in this result is when $X$ contains a non-open
component. 

\section{Invariant Weak Selections}
\label{sec:weak-select-part}

For a partition $\mathscr{P}$ on a set $X$ and $x\in X$, let
$\mathscr{P}[x]\in \mathscr{P}$ be the unique element with
${x\in \mathscr{P}[x]}$. For a space $X$, the \emph{components}
(sometimes called \emph{connected components}) are the maximal
connected subsets of $X$. They form a closed partition $\mathscr{C}$
of $X$, and each element $\mathscr{C}[x]\in \mathscr{C}$ corresponding
to a point $x\in X$ is called the \emph{component} of this point. The
\emph{quasi-component} $\mathscr{Q}[x]$ of a point $x\in X$ is the
intersection of all clopen subsets of $X$ containing this point. The
quasi-components also form a closed partition $\mathscr{Q}$ of $X$,
thus they are simply called the \emph{quasi-components} of $X$.  We
always have $\mathscr{C}[x]\subset \mathscr{Q}[x]$, but the converse
is not necessarily true. However, if $X$ has a continuous weak
selection, then $\mathscr{C}[x]= \mathscr{Q}[x]$ for every $x\in X$,
\cite[Theorem 4.1]{gutev-nogura:00b}. In the sequel, we will freely
rely on this fact.\medskip

In the present section, we will prove the following theorem
dealing with the cws-number of semi-orderable spaces.

\begin{theorem}
  \label{theorem-sel-top-25:1}
  If $X$ is a semi-orderable space, then $\cws(X)\leq 2$. Moreover,
  $\cws(X)=2$ if and only if $X$ is non-orderable and
  has at most 2 components. 
\end{theorem}

In what follows, $\sel_2(X)$ is the collection of all weak selections
for set $X$. Following the standard notations with linear orders, for
$g\in \sel_2(X)$ and subsets $A,B\subset X$, we write $A\leq_g B$ if
$x\leq_g y$ for every $x\in A$ and $y\in B$. Similarly, $A<_g B$ means
that $x<_g y$ for every $x\in A$ and $y\in B$. Moreover, if one of the
sets is a singleton, say $A=\{x\}$, then we will simply write
$x\leq B$ and, respectively, $x<_g B$. Finally, based on this, we will
also consider the following $\leq_g$-open intervals associated to a
(nonempty) subset $A\subset X$.
\begin{equation}
  \label{eq:sel-top-23:2}
  (\gets,A)_{\leq_g}=\{x\in X:x<_g A\}\quad \text{and}\quad
  (A,\to)_{\leq_g}= \{x\in X: A<_g x\}.
\end{equation}
Let us remark that in contrast to the $\leq_g$-open intervals associated
to the points of $X$, the intervals associated to subsets of $X$ may
fail to be open even when $g$ is a continuous weak selection for
$X$. \medskip 

For $g\in \sel_2(X)$, a family $\mathscr{P}$ of subsets of $X$ is
called \emph{$\leq_g$-decisive} or, merely, \emph{$g$-decisive}
\cite{gutev:08a} if $P<_g Q$ or $Q<_g P$ for every
${P,Q\in \mathscr{P}}$ with $P\neq Q$. Evidently, a $g$-decisive
family $\mathscr{P}$ must be pairwise disjoint. The proof of Theorem
\ref{theorem-sel-top-25:1} is based on weak selections which are
``decisive''-invariant with respect to a partition $\mathscr{P}$ of
$X$. Namely, for a partition $\mathscr{P}$ of $X$, we shall say that a
weak selection $g$ for $X$ is \emph{$\mathscr{P}$-invariant} if the
family $\mathscr{P}$ is $g$-decisive. Evidently, such a selection $g$
defines a natural weak selection for $\mathscr{P}$ which is
represented by the same selection relation $\leq_g$, in other words
$g\in \sel_2(\mathscr{P})$. The converse is also true.  Namely, to
each $\sigma\in\sel_2(\mathscr{P})$ and $h_P\in\sel_2(P)$,
$P\in \mathscr{P}$, we may associate a unique $\mathscr{P}$-invariant
weak selection $\sigma* h\in\sel_2(X)$ which is identical to $\sigma$
on the partition $\mathscr{P}$, and identical to $h_P$ on each element
$P\in \mathscr{P}$. For $x,y\in X$, it is defined by
$x<_{\sigma* h} y$ if $\mathscr{P}[x]<_\sigma \mathscr{P}[y]$ and
$\sigma* h(\{x,y\})=h_P(\{x,y\})$ if
$\mathscr{P}[x]=P=\mathscr{P}[y]$. The resulting weak selection
$\sigma*h$ will be found very useful in the proof of Theorem
\ref{theorem-sel-top-25:1}, in fact in destroying transitivity of a
linear order on $X$. To this end, we shall say that a weak selection
$\sigma$ for a set $Z$ is \emph{circular} on $x,y,z\in Z$ if
$x<_\sigma y<_\sigma z<_\sigma x$. In the sequel, every such triple
will be called \emph{$\sigma$-circular} or, merely, \emph{circular}.

\begin{proposition}
  \label{proposition-sel-top-32:1}
  Let $\mathscr{P}$ be a partition of $X$,
  $\mathscr{Q}\subset \mathscr{P}$ be a triple and $\sigma$ be a weak
  selection for $\mathscr{P}$ such that $\mathscr{Q}$ is
  $\sigma$-circular and $P<_\sigma \mathscr{Q}$ or
  $\mathscr{Q}<_\sigma P$, for every $P\in \mathscr{P}$ with
  $P\notin \mathscr{Q}$. If $h_P\in\sel_2(P)$, $P\in \mathscr{P}$,
  then 
  \[
    \bigcup(\gets,\mathscr{Q})_{\leq_\sigma}\in
    \mathscr{T}_{\sigma*h},\quad \mathscr{Q}\subset
    \mathscr{T}_{\sigma* h}\quad \text{and}\quad  
    \bigcup(\mathscr{Q},\to)_{\leq_\sigma}\in \mathscr{T}_{\sigma*h}.
    \]
\end{proposition}

\begin{proof}
  Take weak selections $h_P\in \sel_2(P)$, $P\in \mathscr{P}$, and let
  $g=\sigma* h$ be the corresponding $\mathscr{P}$-invariant weak
  selection for $X$. Also, assume that $\mathscr{Q}=\{Q,R,S\}$ and
  $Q<_\sigma R<_\sigma S<_\sigma Q$. Since we cannot distinguish the
  elements of $\mathscr{Q}$ with respect to the selection relation
  $\leq_\sigma$, it suffices to show that one of them belongs to
  $\mathscr{T}_g$, for instance that $R\in \mathscr{T}_g$. To this
  end, take points $q\in Q$ and $s\in S$.  Since $g$ is
  $\mathscr{P}$-invariant, we have that ${S<_g q<_g R<_g s<_g Q}$ and
  $\{q,s\}<_g P$ or $P<_g \{q,s\}$, whenever
  $P\in \mathscr{P}\setminus \mathscr{Q}$. Consequently,
  $R= (q,s)_{\leq_g}\in \mathscr{T}_g$. To see the other property,
  take also a point $r\in R$. Then $Q<_g r<_g S$ and and $r<_g P$ or
  $P<_g r$, whenever $P\in \mathscr{P}\setminus \mathscr{Q}$. Taking
  in mind that the points $q\in Q$ and $s\in S$ have virtually the
  same property, we get that
  \[
    \bigcup (\gets,\mathscr{Q})_{\leq_\sigma}=\bigcup\{P\in
    \mathscr{P}:P<_\sigma\mathscr{Q}\}=\bigcap_{x\in\{q,r,s\}}(\gets,x)_{\leq_g}\in
    \mathscr{T}_g.
  \]
  Similarly, it follows that
  $\bigcup(\mathscr{Q},\to)_{\leq_\sigma}\in \mathscr{T}_g$, see
  \eqref{eq:sel-top-23:2}.
\end{proof}

We now refine the construction in Proposition
\ref{proposition-sel-top-32:1} extending it to an entire partition
$\mathscr{P}$ of $X$.

\begin{lemma}
  \label{lemma-sel-top-4:1}
  If $\mathscr{P}$ is a partition of $X$ with $|\mathscr{P}|\neq 2$,
  then it has a weak selection $\sigma$ such that
  $\mathscr{P}\subset \mathscr{T}_{\sigma* h}$, for every collection
  of weak selections $h_P\in\sel_2(P)$, $P\in \mathscr{P}$. If
  moreover $\mathscr{P}$ is an open partition of $X$ and each
  $h_P\in\sel_2(P)$, $P\in \mathscr{P}$, is continuous, then
  $\sigma* h$ is also continuous.
\end{lemma}

\begin{proof}
  Let $X$ and $\mathscr{P}$ be as in Lemma \ref{lemma-sel-top-4:1}. If
  $\mathscr{P}$ is a singleton, there is nothing to prove because
  $X\in \mathscr{T}_g$, for any weak selection $g\in
  \sel_2(X)$. Suppose that $|\mathscr{P}|\geq 3$.  If $\mathscr{P}$ is
  infinite, then it has a partition $\Omega$ consisting of triples of
  $\mathscr{P}$. Take a linear order $\leq$ on $\Omega$ and define a
  weak selection $\sigma$ for $\mathscr{P}$ such that each triple
  $\mathscr{Q}\in \Omega$ is $\sigma$-circular and
  $\mathscr{Q}<_\sigma \mathscr{R}$, whenever
  $\mathscr{Q}< \mathscr{R}$ for $\mathscr{Q},\mathscr{R}\in
  \Omega$. In other words, $\sigma$ is the weak selection
  $\ell*\varkappa$, where each
  $\varkappa_\mathscr{Q}\in \sel_2(\mathscr{Q})$ is circular on
  $\mathscr{Q}$, $\mathscr{Q}\in \Omega$, and $\ell\in \sel_2(\Omega)$
  is generated by linear order $\leq$, i.e. $\leq_\ell=\leq$. Thus,
  $\sigma$ is $\Omega$-invariant and, in particular,
  $P<_\sigma \mathscr{Q}$ or $\mathscr{Q}<_\sigma P$ whenever
  $\mathscr{Q}\in \Omega$ and $P\in \mathscr{P}\setminus
  \mathscr{Q}$. Accordingly, by Proposition
  \ref{proposition-sel-top-32:1}, $\sigma$ is as required.\smallskip

  The remaining case is when the partition $\mathscr{P}$ is finite. In
  this case, take a maximal pairwise disjoint family $\Omega$ of
  triples of $\mathscr{P}$ and set
  $\mathscr{A}= \mathscr{P}\setminus \bigcup\Omega$. Evidently,
  $\Omega\neq \emptyset$ because $|\mathscr{P}|\geq 3$. Let $\leq$ be
  a linear order on $\Omega$, and $\eta$ be a weak selection for
  $\bigcup\Omega=\mathscr{P}\setminus \mathscr{A}$ such that each
  $\mathscr{Q}\in \Omega$ is $\eta$-circular and
  $\mathscr{Q}<_\eta \mathscr{R}$, whenever $\mathscr{Q}< \mathscr{R}$
  for $\mathscr{Q},\mathscr{R}\in \Omega$.  If
  $\mathscr{A}=\emptyset$, then $\sigma=\eta$ is a weak selection for
  $\mathscr{P}$ and, by Proposition \ref{proposition-sel-top-32:1}, it
  is as required. Suppose that $\mathscr{A}\neq \emptyset$, and take
  weak selections $h_P\in\sel_2(P)$, $P\in \mathscr{P}$.  We will
  construct the required weak selection
  $\sigma\in \sel_2(\mathscr{P})$ as an extension of $\eta$ on the
  elements of $\mathscr{A}$, i.e.\
  $\sigma\uhr \mathscr{F}_2(\mathscr{P}\setminus
  \mathscr{A})=\eta$. \smallskip

  \textbf{(A)} If $\mathscr{A}=\{A\}$ is a singleton, the extension
  $\sigma$ can be defined by $A<_\sigma P$, for every
  $P\in \mathscr{P}\setminus \mathscr{A}$. Evidently, the collection
  $\{\mathscr{A}\}\cup \Omega$ is $\sigma$-decisive and each triple
  $\mathscr{Q}\in \Omega$ is $\sigma$-circular. Hence, by Proposition
  \ref{proposition-sel-top-32:1},
  $\mathscr{Q}\subset \mathscr{T}_{\sigma*h}$ for every
  $\mathscr{Q}\in \Omega$. For the same reason,
  $A\in \mathscr{T}_{\sigma*h}$ because
  $(\gets, \mathscr{H})_{\leq_\sigma}=\{A\}=\mathscr{A}$, where
  $\mathscr{H}\in \Omega$ is the $\leq$-minimal element of $\Omega$,
  see \eqref{eq:sel-top-23:2}.
  \smallskip

  \textbf{(B)} If $\mathscr{A}$ has two elements $A$ and $B$, then the
  extension $\sigma$ can be defined by $A<_\sigma P<_\sigma B$ for
  $P\in \mathscr{P}\setminus \mathscr{A}$. Now the family
  $\{\{A\}\}\cup \Omega\cup \{\{B\}\}$ is $\sigma$-decisive and as in
  \textbf{(A)}, $A\in \mathscr{T}_{\sigma*h}$ and
  $\mathscr{Q}\subset \mathscr{T}_{\sigma*h}$,
  $\mathscr{Q}\in \Omega$. Similarly, since $\Omega$ also has a
  $\leq$-maximal element $\mathscr{R}\in \Omega$, by Proposition
  \ref{proposition-sel-top-32:1},
  $B=\bigcup(\mathscr{R},\to)_{\leq_\sigma}\in
  \mathscr{T}_{\sigma*h}$.
\end{proof}

We are now ready to finalise the proof of Theorem
\ref{theorem-sel-top-25:1}.

\begin{proof}[Proof of Theorem \ref{theorem-sel-top-25:1}]
  If $X$ is semi-orderable, then it is a topological sum of two
  orderable spaces $X_0$ and $X_1$. Take a compatible linear order
  $\leq_i$ on $X_i$, $i=0,1$, and let $\leq$ be the linear order on
  $X$ generated by $\leq_0$, $\leq_1$ and $X_0<X_1$. Then $\leq$ is a
  compatible linear order on $X$, called \emph{canonical} in
  \cite{gutev:07b}. Hence, it corresponds to a continuous weak
  selection $g$ for $X$ with $\leq_g=\leq$. Similarly, take another
  continuous weak selection $h$ for $X$ corresponding to the canonical
  order $\leq_h$ on $X$ determined by $\leq_0$, $\leq_1$ and
  $X_1<_h X_0$. Then $X_0,X_1\in \mathscr{T}_g\vee \mathscr{T}_h$ and,
  therefore, $\mathscr{T}_g\vee \mathscr{T}_h$ is the topology of
  $X$. Accordingly, $\cws(X)\leq 2$. Suppose that the semi-orderable
  space $X$ is not orderable and has at most two components. Then $X$
  has precisely two components, consequently these components are
  $X_0$ and $X_1$. Take any continuous weak selection $\sigma$ for
  $X$. Then by \cite[Lemma 7.2]{michael:51}, $\leq_\sigma$ is a linear
  order on $X_i$, $i=0,1$. Moreover, by \cite[Proposition
  2.6]{gutev:07a}, $X_0<_\sigma X_1$ or $X_1<_\sigma X_0$. This
  implies that $\leq_\sigma$ is a linear order on $X$. Since $X$ is
  not orderable, $\mathscr{T}_\sigma$ is not the topology of
  $X$. Thus, $\cws(X)=2$. \smallskip

  Suppose finally that $X$ is a semi-orderable space which has at
  least 3 different components $C_1,C_2,C_3\in \mathscr{C}[X]$. Since
  the components of $X$ coincide with the quasi-components, there are
  disjoint clopen sets $U_1,U_2\subset V$ such that $C_1\subset U_1$,
  $C_2\subset U_2$ and $C_3\subset U_3=X\setminus (U_1\cup
  U_2)$. According to \cite[Corollary 4.10]{gutev:07b}, each $U_i$,
  $i\leq 3$, is itself semi-orderable being open in $X$. This implies
  that $X$ has a clopen partition $\mathscr{Z}$ of orderable spaces
  with $|\mathscr{Z}|\geq 3$.  For each $Z\in \mathscr{Z}$, take a
  linear order $\leq_Z$ generating the topology of $Z$, and denote by
  $h_Z$ the (continuous) weak selection for $Z$ with
  $\leq_{h_Z}=\leq_Z$. Finally, let $\sigma$ be a weak selection for
  $\mathscr{Z}$ as in Lemma \ref{lemma-sel-top-4:1}. Then
  $\mathscr{T}_{\sigma* h}$ is the topology of $X$. Indeed, by Lemma
  \ref{lemma-sel-top-4:1},
  $\mathscr{Z}\subset \mathscr{T}_{\sigma* h}$, so each
  $Z\in \mathscr{Z}$ is open in $\mathscr{T}_{\sigma* h}$. Moreover,
  the restriction of $\sigma* h$ on each $Z\in \mathscr{Z}$ is
  identical to $h_Z$ and generates the topology of $Z$. The proof is
  complete.
\end{proof}

We conclude this section with a refinement of Lemma
\ref{lemma-sel-top-4:1} which will be found useful in resolving the
cws-number of metrizable suborderable spaces.

\begin{lemma}
  \label{lemma-sel-top-32:1}
  Let $\mathscr{P}$ be a partition of $X$ with $|\mathscr{P}|\geq 7$,
  and $S\in \mathscr{P}$ be a fixed element. Then $\mathscr{P}$ has a
  weak selection $\sigma$ such that $S$ is the $\leq_\sigma$-maximal
  element of $\mathscr{P}$ and
  $\mathscr{P}\subset \mathscr{T}_{\sigma* h}$, for every collection
  of weak selections $h_P\in\sel_2(P)$, $P\in \mathscr{P}$.
\end{lemma}

\begin{proof}
  Take weak selections $h_P\in\sel_2(P)$, $P\in \mathscr{P}$, and set
  $\mathscr{P}_0=\mathscr{P}\setminus\{S\}$. The proof consists of
  constructing a weak selection $\eta\in\sel_2(\mathscr{P}_0)$ and an
  $\eta$-circular triple $\mathscr{R}\subset \mathscr{P}_0$ such that
  $\mathscr{P}_0\subset \mathscr{T}_{\eta* h}$ and
  $P<_{\eta} \mathscr{R}$ for every
  $P\in \mathscr{P}_0\setminus \mathscr{R}$. Once this is done, one
  can extend $\eta$ to a weak selection $\sigma$ for the entire
  partition $\mathscr{P}$ by letting $P<_\sigma S$, for every
  $P\in \mathscr{P}_0$. Then it follows from Proposition
  \ref{proposition-sel-top-32:1} that
  $\mathscr{P}\subset \mathscr{T}_{\sigma*h}$. Indeed, $S\in
  \mathscr{T}_{\sigma*h}$ and 
  $\mathscr{R}\subset \mathscr{T}_{\sigma*h}$ because $\mathscr{R}$ is
  $\sigma$-circular and $P<_\sigma \mathscr{R}<_\sigma S$, for
  every $P\in \mathscr{P}_0\setminus \mathscr{R}$. For the same
  reason, $\bigcup (\mathscr{P}_0\setminus \mathscr{R})\in
  \mathscr{T}_{\sigma*h}$. Since $\mathscr{P}_0\subset
  \mathscr{T}_{\eta* h}$ and $\sigma$ is an extension of $\eta$, we
  also have that $\mathscr{P}_0\setminus \mathscr{R}\subset
  \mathscr{T}_{\sigma*h}$.\smallskip 

  The construction of the selection $\eta\in \sel_2(\mathscr{P}_0)$
  and the triple $\mathscr{R}\subset \mathscr{P}_0$ is very similar to
  that in Lemma \ref{lemma-sel-top-4:1}. Namely, if $\mathscr{P}_0$ is
  infinite and $\Omega$ is a partition of $\mathscr{P}_0$ consisting
  of triples of $\mathscr{P}_0$, then we can take a linear order
  $\leq$ on $\Omega$ with respect to which it has a maximal element
  $\mathscr{R}\in \Omega$. Next, we can define the required weak
  selection $\eta\in \sel_2(\mathscr{P}_0)$ precisely as in the proof
  of Lemma \ref{lemma-sel-top-4:1}. Suppose that $\mathscr{P}_0$ is
  finite, $\Omega$ is a maximal family of triples of $\mathscr{P}_0$
  and $\leq$ is a linear order on $\Omega$. The required triple
  $\mathscr{R}\in \Omega$ is now the $\leq$-maximal element of
  $\Omega$. As for the selection $\eta\in \sel_2(\mathscr{P}_0)$, its
  construction depends on the set
  $\mathscr{A}=\mathscr{P}_0\setminus \bigcup \Omega$. If
  $\mathscr{A}=\emptyset$ or $\mathscr{A}$ is a singleton, the
  construction is identical to that in Lemma \ref{lemma-sel-top-4:1}.
  The case when $\mathscr{A}$ has two elements $A,B\in \mathscr{A}$
  requires a slight modification in the construction of \textbf{(B)}
  in the proof of that lemma.  Briefly, define a weak selection
  $\eta_0\in \sel_2(\bigcup\Omega)$ such that each element of $\Omega$
  is $\eta_0$-circular, and $\mathscr{G}<_{\eta_0} \mathscr{H}$,
  whenever $\mathscr{G},\mathscr{H}\in \Omega$ with
  $\mathscr{G}<\mathscr{H}$. Take the $\leq$-minimal triple
  $\mathscr{Q}\in \Omega$, and extend $\eta_0$ to a weak selection
  $\eta\in\sel_2(\mathscr{P}_0)$ by
  $A<_\eta \mathscr{Q}<_\eta B<_\eta \mathscr{G}$ and
  $A<_\eta \mathscr{G}$, for every $\mathscr{G}\in \Omega$ with
  $\mathscr{G}\neq \mathscr{Q}$. Then by Proposition
  \ref{proposition-sel-top-32:1}, $A\in \mathscr{T}_{\eta*h}$ and
  $\mathscr{G}\subset \mathscr{T}_{\eta*h}$ for every
  $\mathscr{G}\in \Omega$, see \textbf{(A)} in the proof of Lemma
  \ref{lemma-sel-top-4:1}. To show finally that
  $B\in \mathscr{T}_{\eta*h}$, observe that the $\leq$-minimal element
  $\mathscr{Q}\in \Omega$ is not equal to the $\leq$-maximal element
  of $\Omega$ because $|\mathscr{P}_0|\geq 6$.  Let
  $\mathscr{H}\in \Omega$ be the $\leq$-minimal element of $\Omega$
  with $B<_\eta \mathscr{H}$. Then
  $\mathscr{Q}<_\eta B<_\eta \mathscr{H}$ and, in fact,
  $\{B\}=(\mathscr{Q},\to)_{\leq_\eta}\cap
  (\gets,\mathscr{H})_{\leq_\eta}$.  Hence, by Proposition
  \ref{proposition-sel-top-32:1}, $B\in \mathscr{T}_{\eta*h}$.
\end{proof}

The following is an immediate consequence of Lemma \ref{lemma-sel-top-32:1}. 

\begin{corollary}
  \label{corollary-sel-top-14:1}
  Let $\mathscr{P}$ be a partition of $X$, $S\in \mathscr{P}$ and
  $\mathscr{P}\setminus \{S\}=\mathscr{Q}\cup \mathscr{R}$ for some
  disjoint subsets $\mathscr{Q},\mathscr{R}\subset \mathscr{P}$.  If
  $|\mathscr{Q}|\geq 6$ and $|\mathscr{R}|\geq 6$, then $\mathscr{P}$
  has a weak selection $\sigma$ such that
  $\mathscr{Q}\cup\{S\}\leq_\sigma \{S\}\cup \mathscr{R}$ and
  $\mathscr{P}\subset \mathscr{T}_{\sigma* h}$, for every collection
  of weak selections $h_P\in\sel_2(P)$, $P\in \mathscr{P}$.
\end{corollary}

\section{Suborderable Spaces and Components}

Let $\mathscr{P}$ be a partition $X$. In this section, it will make
sense to look at $\mathscr{P}$ as a map from $X$ to the subsets of $X$
assigning to each $x\in X$ the unique element
$\mathscr{P}[x]\in \mathscr{P}$ with $x\in \mathscr{P}[x]$. In this
interpretation, $\mathscr{P}[X]=\{\mathscr{P}[x]:x\in X\}$ is the
partition $\mathscr{P}$. As a topological space, we will consider
$\mathscr{P}=\mathscr{P}[X]$ endowed with quotient topology generated
by $\mathscr{P}$ as an equivalence relation on $X$, so
$\mathscr{P}:X\to \mathscr{P}[X]$ is the corresponding quotient
map. Then a subset $\mathscr{U}\subset \mathscr{P}[X]$ is open if and
only if $\mathscr{P}^{-1}(\mathscr{U})=\bigcup\mathscr{U}$ is open in
$X$. Two natural closed partitions on a space $X$ are the components
$\mathscr{C}[X]$ and the quasi-components $\mathscr{Q}[X]$. As
mentioned in the previous section, $\mathscr{C}[X]=\mathscr{Q}[X]$
provided $X$ has a continuous weak selection. \medskip

A space $Z$ is \emph{zero-dimensional} if it has a base of clopen
sets, and is \emph{strongly zero-dimensional} if $\dim(Z)=0$, where
$\dim(Z)$ is the \emph{covering dimension}. In the realm of normal
spaces, $\dim(Z)=0$ if and only if the \emph{large inductive
  dimension} of $Z$ is 0, i.e.\ if every two disjoint closed subsets
of $Z$ are contained in disjoint clopen subsets. The following theorem
will play a crucial role in this paper.

\begin{theorem}
  \label{theorem-sel-top-9:1}
  If $X$ is a suborderable space, then $\mathscr{C}[X]$ is a strongly
  zero-dimen\-sional suborderable space and $\mathscr{C}:X\to
  \mathscr{C}[X]$ is closed. Moreover, $\mathscr{C}[X]$ is metrizable
  provided so is $X$.
\end{theorem}

The proof of Theorem \ref{theorem-sel-top-9:1} is based on the
following considerations.\medskip

A subset $S\subset X$ of an ordered set $(X,\leq)$ is called
\emph{convex} (also, an \emph{interval}) if
$[x,y]_{\leq}=\{z\in X: x\leq z\leq y\}\subset S$ for every $x,y\in S$
with $x\leq y$.  E.~\v{C}ech \cite{cech:66} showed that a space $X$ is
\emph{suborderable} (a subspaces of an orderable space) if and only if
it admits a linear ordering such that the corresponding open interval
topology is coarser than the topology of $X$, and $X$ has a base of
convex sets with respect to this order. In the sequel, every such
linear order $\leq$ on $X$ will be called \emph{compatible}, and we
will often say that $X$ is suborderable with respect to $\leq$ or,
simply, that $(X,\leq)$ is suborderable.\medskip

For a suborderable space $(X,\leq)$ and a (nonempty) subset
$A\subset X$, let $(\gets,A)_{\leq}$ and $(A,\to)_{\leq}$ be defined
as in \eqref{eq:sel-top-23:2}. Evidently, $(\gets,A)_{\leq}$ and
$(A,\to)_\leq$ are convex but not necessarily open. Here, we will also
need the following ``$\leq$-closed'' intervals associated to $A$:
\begin{equation}
  \label{eq:sel-top-34:1}
  (\gets,A]_{\leq}=X\setminus (A,\to)_\leq\quad \text{and}\quad
  [A,\to)_\leq=X\setminus (\gets,A)_\leq.
\end{equation}
Evidently, $(\gets,A]_{\leq}$ and $[A,\to)_\leq$ are also convex and
contain the set $A$. In fact, we need these intervals in the special
case when $A=C\in \mathscr{C}[X]$ is a component of $X$.  In this
case, since $C$ is connected, it follows that $x< C$ or $C< x$, for
every $x\in X\setminus C$. Hence, both $(\gets,C)_\leq$ and
$(C,\to)_\leq$ are open in $X$ because
$(\gets, C)_\leq=(\leftarrow,y)_{\leq}\setminus C$ and
$(C,\to)_\leq=(y,\to)_{\leq}\setminus C$, whenever $y\in C$. In
particular, $(\gets,C]_{\leq}$ and $[C,\to)_{\leq}$ are closed sets
with $(\gets,C]_{\leq}\cap [C,\to)_{\leq}=C$.

\begin{proposition}
  \label{proposition-sel-top-18:2}
  If $(X,\leq)$ is a suborderable space, $C$ is a component of $X$ and
  $U$ is a neighbourhood of $C$, then there exists a convex open set
  $O\subset X$ such that $C\subset O\subset U$ and
  $O=\mathscr{C}^{-1}(\mathscr{C}(O))$. In particular,
  $\mathscr{C}:X\to \mathscr{C}[X]$ is a closed map.
\end{proposition}

\begin{proof}
  The set $O$ can be defined as the intersection $L\cap R$ of two open
  convex sets $L,R\subset X$ which are unions of components of
  $X$. For instance, take $L=(\gets,C]_\leq$ if this set is
  open. Otherwise, if $(\gets,C]_\leq$ is not open, the interval
  $(C,\to)_\leq$ is not closed and there exists a point
  $p\in (\gets,C]_\leq\cap \overline{(C,\to)_\leq}$. Evidently,
  $p\in C\subset U$ is the last element of $C$, which is a
  non-isolated point of $X\setminus C$. Hence, there exists $q\in X$
  such that $p< q$ and $[p,q]_\leq \subset U$. In this case, the
  component $S=\mathscr{C}[q]$ is different from $C$, in fact $C<
  S$. Now, we can take $L=(\gets, S)_\leq$. Since the construction of
  $R$ is completely analogous, the proof is complete.
\end{proof}

Given any two different components $C,S\subset X$, we have that either
$C< S$ or $S< C$. In particular, $\leq$ is a linear order on the
partition $\mathscr{C}[X]$ and, in fact, the quotient space
$\mathscr{C}[X]$ is suborderable with respect to $\leq$.

\begin{corollary}
  \label{corollary-sel-top-18:1}
  If $(X,\leq)$ is a suborderable space, then the quotient space
  $\mathscr{C}[X]$ is also suborderable with respect to $\leq$.
\end{corollary}

\begin{proof}
  If $C$ is a component of $X$ and
  $\mathscr{L}_C=\{S\in \mathscr{C}[X]: S< C\}$, then
  $\mathscr{C}^{-1}(\mathscr{L}_C)$ is open in $X$ because
  ${\mathscr{C}^{-1}(\mathscr{L}_C)=(\gets,C)_{\leq}}$. Thus,
  $\mathscr{L}_C$ is open in $\mathscr{C}[X]$. Similarly, the set
  $\mathscr{R}_C=\{S\in \mathscr{C}[X]:C< S\}$ is also open in
  $\mathscr{C}[X]$. Hence, the topology of $\mathscr{C}[X]$ is finer
  than the open interval one generated by $\leq$. In fact,
  $\mathscr{C}[X]$ is suborderable with respect to $\leq$ because it
  has a base of convex sets with respect to $\leq$. Namely, if
  $U\subset X$ is an open set containing $C$, then by Proposition
  \ref{proposition-sel-top-18:2}, there exists a convex open set
  $O\subset X$ such that $C\subset O\subset U$ and
  $O=\mathscr{C}^{-1}(\mathscr{C}(O))$. The proof is complete.
\end{proof}

A space $Z$ is \emph{totally disconnected} if each point of $Z$ is an
intersection of clopen sets or, equivalently, if
$\mathscr{Q}[z]=\{z\}$ for every $z\in Z$.  It was shown in
\cite[Lemma 1]{herrlich:65} that every orderable totally disconnected
space is strongly zero-dimensional. Subsequently, it was remarked by
Purisch \cite[Proposition 2.3]{purisch:77} that the same argument
works to show that every totally disconnected suborderable space is
strongly zero-dimensional. For a suborderable space $X$, the
components of the quotient space $\mathscr{C}[X]$ are
singletons. According to Corollary \ref{corollary-sel-top-18:1},
$\mathscr{C}[X]$ is suborderable as well. Hence, as mentioned above,
it is also strongly zero-dimensional. Thus, we have also the following
consequence, the second part of which follows from Proposition
\ref{proposition-sel-top-18:2}.

\begin{corollary}
  \label{corollary-sel-top-10:2}
  If $X$ is a suborderable space, then $\mathscr{C}[X]$ is strongly
  zero-dimen\-sional. In particular, if $C\in \mathscr{C}[X]$ and
  $U\subset X$ is an open set containing $C$, then there exists a
  clopen set $V\subset X$ with $C\subset V\subset U$.
\end{corollary}

To conclude the preparation for the proof of Theorem
\ref{theorem-sel-top-9:1}, let us remark that a possible way to model
the quotient space $\mathscr{C}[X]$ for a suborderable space $X$ is by
a suitable subset $Z\subset X$. For instance, for every
$C\in \mathscr{C}[X]$ one can take a point $x_C\in C$ and consider
$Z=\{x_C:C\in \mathscr{C}[X]\}$. The set $Z$ looks similar to
$\mathscr{C}[X]$ utilising the idea to identify each component
$C\in \mathscr{C}[X]$ into a single point.  However, such a simplistic
approach may not lead even to a closed subspace $Z\subset X$. Another
interesting approach was offered by Purisch \cite{purisch:77}. To this
end, recall that a point $p$ of a connected space $C$ is called
\emph{cut} if $C\setminus\{p\}$ is not connected; and it is
\emph{noncut} if $C\setminus\{p\}$ is connected. In \cite{purisch:77},
Purisch considered spaces $X$ each of whose components has at most two
noncut points; clearly, suborderable spaces have this property. Then
he defined a subset $Z\subset X$ as follows.  If a component
$C\in \mathscr{C}[X]$ is a singleton or open, then one point of $C$
belongs to $Z$; if $C\in \mathscr{C}[X]$ is a non-degenerate non-open
component of $X$, then two points of $C$, including all noncut points,
belong to $Z$. Such sets and some slight modifications of them were
called \emph{Purisch sets} in \cite{gutev:07a,gutev-nogura:09a}. One
of the best properties of Purisch sets is that they are closed in $X$,
moreover every two such sets are homeomorphic. Herewith, we are mainly
interested in the fact that every suborderable space $X$ contains a
closed subset $Z\subset X$ with $1\leq |Z\cap C|\leq 2$, for every
$C\in \mathscr{C}$. Indeed, every Purisch set has this
property. Hence, we have the following observation.

\begin{proposition}
  \label{proposition-sel-top-10:1}
  Every suborderable space $X$ contains a closed subset $Z\subset X$
  such that $Z\cap C$ is a nonempty finite set for every $C\in
  \mathscr{C}[X]$.
\end{proposition}

We are now ready to prove Theorem \ref{theorem-sel-top-9:1}.

\begin{proof}[Proof of Theorem \ref{theorem-sel-top-9:1}]
  According to Proposition \ref{proposition-sel-top-18:2} and
  Corollaries \ref{corollary-sel-top-18:1} and
  \ref{corollary-sel-top-10:2}, it remains to show that
  $\mathscr{C}[X]$ is metrizable if so is $X$. So, suppose that $X$ is
  a metrizable suborderable space. Also, let $Z\subset X$ be as in
  Proposition \ref{proposition-sel-top-10:1}. Since $Z$ is closed in
  $X$ and $\mathscr{C}:X\to \mathscr{C}[X]$ is a closed map, so is the
  restriction $g=\mathscr{C}\uhr Z:Z\to \mathscr{C}[X]$. Moreover, $g$
  is surjective because $Z\cap C\neq \emptyset$,
  $C\in \mathscr{C}[X]$, and $g$ is perfect because each $Z\cap C$,
  $C\in \mathscr{C}[X]$, is compact being finite. Thus,
  $\mathscr{C}[X]$ is a perfect image of a metrizable space and is
  itself metrizable, see \cite[Theorem 4.4.15]{engelking:89}. The
  proof of Theorem \ref{theorem-sel-top-9:1} is complete.
\end{proof}

In fact, Theorem \ref{theorem-sel-top-9:1} will be used implicitly in
the setting of the following consequence of it.

\begin{corollary}
  \label{corollary-sel-top-14:3}
  Let $X$ be a metrizable suborderable space and $\mathscr{U}$ be
  an open cover of $X$ such that each component of $X$ is contained in
  some element of $\mathscr{U}$. Then $\mathscr{U}$ has a discrete
  refinement. 
\end{corollary}

\begin{proof}
  For every component $C\in \mathscr{C}[X]$ there exists
  $U_C\in \mathscr{U}$ with $C\subset U_C$. By Proposition
  \ref{proposition-sel-top-18:2}, each $C\in \mathscr{C}[X]$ is
  contained in an open subset $\mathscr{V}_C\subset \mathscr{C}[X]$
  such that $C\subset \mathscr{C}^{-1}(\mathscr{V}_C)\subset
  U_C$. Thus, $\big\{\mathscr{V}_C:C\in \mathscr{C}[X]\big\}$ is an
  open cover of $\mathscr{C}[X]$. By Theorem
  \ref{theorem-sel-top-9:1}, $\mathscr{C}[X]$ is a strongly
  zero-dimensional metrizable space, so there exists a discrete
  cover $\Omega$ of $\mathscr{C}[X]$ which refines
  $\big\{\mathscr{V}_C:C\in \mathscr{C}[X]\big\}$. The cover
  $\mathscr{C}^{-1}(\Omega) =\left\{\mathscr{C}^{-1}(\mathscr{W}):
    \mathscr{W}\in \Omega\right\}$ of $X$ is as required.
\end{proof}

\section{Partitions of Suborderable
  Metrizable Spaces}  

A metric $\rho$ on an ordered set $(X,\leq)$ is called \emph{convex}
if $\rho(a,b)\leq \rho(x,y)$, whenever $x,y,a,b\in X$ with
$x\leq a\leq b\leq y$, see \cite{purisch:77}. Equivalently, $\rho$ is
convex iff $\max\{\rho(x,y),\rho(y,z)\}\leq \rho(x,z)$, whenever
$x,y,z\in X$ with $x\leq y\leq z$, \cite[Proposition
2.4]{purisch:77}. According to \cite[Proposition
2.5]{purisch:77}, each metrizable suborderable space $(X,\leq)$ admits
a compatible convex metric $\rho$.\medskip 

Throughout this section, $(X,\leq)$ is a metrizable suborderable space
and $\rho$ is a fixed convex metric on $X$ compatible with its
topology. For $x\in X$ and $\varepsilon>0$, we will use
$\mathbf{O}(x,\varepsilon)$ for the \emph{open $\varepsilon$-ball}
centred at $x$. Moreover, for a subset $A\subset X$, let
$\mathbf{O}(A,\varepsilon)=\bigcup_{x\in A}\mathbf{O}(x,\varepsilon)$
and $\diam(A)$ be the \emph{diameter} of $A$. In case $A=C$ is a
component of $X$, another ``$\varepsilon$-neighbourhood'' of $C$ will
play an important role. Namely, to each component
$C\in \mathscr{C}[X]$, we associate the numbers
$\ell(C),r(C)\in \{0,1\}$ defined by
\begin{equation}
  \label{eq:sel-top-22:1}
  \ell(C)=\left|\overline{(\gets,C)_\leq}
    \cap C\right|\quad\text{and}\quad
  r(C)=\left|C\cap \overline{(C,\to)_\leq}\right|.
\end{equation}
These numbers simply indicate whether the ``$\leq$-open'' intervals
associated to $C$ are also closed, i.e.\ clopen, see
\eqref{eq:sel-top-23:2} and \eqref{eq:sel-top-34:1}. Namely,
$\ell(C)=0$ precisely when $(\gets,C)_\leq$ is clopen, equivalently
when $[C,\to)_\leq=X\setminus (\gets,C)_\leq$ is clopen. Similarly,
for the number $r(C)$. Based on this, for $C\in \mathscr{C}[X]$ and
$\varepsilon>0$, we set
\begin{equation}
  \label{eq:sel-top-23:1}
  \Delta(C,\varepsilon)=
  \begin{cases}
    C &\text{if $\ell(C)=0=r(C)$,}\\
    \mathbf{O}\left(C,\frac\varepsilon2\right)\cap (\gets,C]_\leq
    &\text{if
      $\ell(C)\neq 0= r(C)$,}\\
    \mathbf{O}\left(C,\frac\varepsilon2\right)\cap [C,\to)_\leq
    &\text{if
      $\ell(C)=0\neq r(C)$,}\\
    \mathbf{O}\left(C,\frac\varepsilon2\right) &\text{if
      $\ell(C)\neq 0\neq r(C)$.}\\
  \end{cases}
\end{equation}

Evidently, $\Delta(C,\varepsilon)$ is an open set containing $C$ and
contained in $\mathbf{O}\left(C,\frac\varepsilon2\right)$. Here are
some other properties of these neighbourhoods.

\begin{proposition}
  \label{proposition-sel-top-34:1}
  If $C\in \mathscr{C}[X]$ and $U\subset X$ is an open set with
  $C\subset U$, then there exists $\varepsilon>0$ such that
  $\Delta(C,\varepsilon)\subset U$. 
\end{proposition}

\begin{proof}
  If $\ell(C)=0=r(C)$, this is obvious because
  $\Delta(C,\varepsilon)=C\subset U$, for every $\varepsilon>0$, see
  \eqref{eq:sel-top-23:1}. If $\ell(C)=1$, let $p\in C$ be the first
  element of $C$, i.e.\ $p\leq C$, and $\varepsilon>0$ be such that
  $\mathbf{O}\left(p,\frac\varepsilon2\right)\subset U$. If
  $x\in \Delta(C,\varepsilon)$ with $x<C$, then
  $x\in \mathbf{O}\left(C,\frac\varepsilon2\right)$ and therefore
  $\rho(x,y)<\frac\varepsilon2$ for some $y\in C$. Since the metric
  $\rho$ is convex and $x<p\leq y$, we get that
  $x\in \mathbf{O}\left(p,\frac\varepsilon2\right)$. Thus, by
  \eqref{eq:sel-top-23:1},
  $\Delta(C,\varepsilon)\subset
  \mathbf{O}\left(p,\frac\varepsilon2\right)\cup C\subset U$ provided
  $r(C)=0$. The case $\ell(C)=0\neq r(C)$ is completely identical, it
  follows by applying the same argument with the last element of $C$.
  Similarly, the property follows when $\ell(C)=1=r(C)$.
\end{proof}

\begin{proposition}
  \label{proposition-sel-top-14:2}
  If $C\in \mathscr{C}[X]$ and $\varepsilon>0$, then
  $\diam(S)\leq\frac\varepsilon2$ for any other component
  $S\in \mathscr{C}[X]$ with $S\subset \Delta(C,\varepsilon)$.
\end{proposition}

\begin{proof}
  Suppose that $S\in \mathscr{C}[X]\setminus\{C\}$ with
  $S\subset \Delta(C,\varepsilon)$. Then
  $S\subset \mathbf{O}\left(C,\frac\varepsilon2\right)$ and $S<C$ or
  $C<S$.  If $S<C$ and $x,y\in S$ with $x\leq y$, then there exists
  $z\in C$ such that $\rho(x,z)<\frac\varepsilon2$. Since $\rho$ is
  convex, this implies that $\rho(x,y)<\frac\varepsilon2$. Therefore,
  $\diam(S)\leq \frac\varepsilon2$. Similarly,
  $\diam(S)\leq\frac\varepsilon2$ provided $C<S$.
\end{proof}

If each component of $X$ is open, then $X$ is a topological sum of its
components, hence it is also semi-orderable, see \cite[Theorem
4.2]{gutev:07b}. Accordingly, the cws-number of $X$ is completely
resolved by Theorem \ref{theorem-sel-top-25:1}. Thus, the remaining
case for $\cws(X)$ is when $X$ has a component which is not open in
$X$. In particular, in this case, $X$ has infinitely many
components. In what follows, we will place this further restriction on
$X$, namely that it has \emph{infinitely many components}.\medskip

If $U\subset X$ is a clopen set, then
$\mathscr{C}[U]=\{C\in \mathscr{C}[X]: C\cap U\neq \emptyset\}$. Here,
an important role will be played by the collection $\mathscr{D}[X]$ of
all clopen subsets $U\subset X$ such that
\begin{equation}
  \label{eq:sel-top-27:1}
  U\in \mathscr{C}[X]\quad \text{or}\quad |\mathscr{C}[U]|\geq
  \omega. 
\end{equation}
The collection $\mathscr{D}[X]$ is in good accord with the
neighbourhoods of the components defined in \eqref{eq:sel-top-23:1}.

\begin{proposition}
  \label{proposition-sel-top-27:1}
  If $U\subset X$ is a clopen set such that
  $C\subset U\subset\Delta(C,\varepsilon)$ for some
  $C\in \mathscr{C}[X]$ and $\varepsilon>0$, then
  $U\in \mathscr{D}[X]$. Moreover, each $U\in \mathscr{D}[X]$ with
  ${|\mathscr{C}[U]|\geq \omega}$, has an open partition
  $U_1,U_2\in \mathscr{D}[X]$.
\end{proposition}

\begin{proof}
  The first part is an immediate consequence of
  \eqref{eq:sel-top-22:1} and \eqref{eq:sel-top-23:1}. Take an element
  $U\in \mathscr{D}[X]$ with $|\mathscr{C}[U]|\geq \omega$. If $U$
  contains a clopen component $C\in \mathscr{C}[X]$, then
  $U_1=C\in \mathscr{D}[X]$ and
  $U_2=U\setminus C\in \mathscr{D}[X]$ form a partition of
  $U$. Otherwise, if $U$ doesn't contain an open component, take a
  nonempty clopen set $U_1\subset U$ with
  $U_2=U\setminus U_1\neq \emptyset$, which is possible because the
  components of $X$ coincide with the quasi-components. Then each
  $U_i$, $i=1,2$, contains infinitely many components, so
  $U_1,U_2\in \mathscr{D}[X]$.
\end{proof}

For $\varepsilon>0$ and a clopen set $U\subset X$, let
$\mathscr{C}_\varepsilon[U]$ be the family of all components of $U$ of
diameter $\geq \varepsilon$, and $\mathscr{C}_*[U]$ --- all
non-degenerate components of $U$, i.e.\
\begin{equation}
  \label{eq:sel-top-29:1}  
  \mathscr{C}_{\varepsilon}[U]=\{C\in \mathscr{C}[U]: \diam(C)\geq
  \varepsilon\}\quad\text{and}\quad
  \mathscr{C}_*[U]=\bigcup_{\varepsilon>0}\mathscr{C}_\varepsilon[U]. 
\end{equation}
Based on this, we also define the following subfamily of
$\mathscr{D}[X]$, namely
  \begin{multline}
  \label{eq:sel-top-29:2}
  \mathscr{D}[X,\varepsilon]=\big\{U\in \mathscr{D}[X]:
  \diam(U)<\varepsilon\quad \text{or}\\ 
  U\subset\Delta(C,\varepsilon),\ \text{for some $C\in
    \mathscr{C}_\varepsilon[U]$}\big\}.
\end{multline}

\begin{lemma}
  \label{lemma-sel-top-14:1}
  Whenever $\varepsilon>0$, each $Z\in \mathscr{D}[X]$ has a partition
  $\mathscr{U}\subset \mathscr{D}[X,\varepsilon]$ such that
  $|\mathscr{U}|\neq 2$.
\end{lemma}

\begin{proof}
  If $Z\in \mathscr{D}[X,\varepsilon]$, simply take
  $\mathscr{U}=\{Z\}$. Suppose that
  $Z\notin \mathscr{D}[X,\varepsilon]$, and take a point $z\in
  Z$. Then $Z$ contains a clopen set $V$ such that
  $\mathscr{C}[z]\subset V\subset\Delta(\mathscr{C}[z],\varepsilon)$,
  see Corollary \ref{corollary-sel-top-10:2}. Hence, by Proposition
  \ref{proposition-sel-top-14:2}, $V\cap C=\emptyset$ for every
  $C\in \mathscr{C}_\varepsilon[X]$ with $C\neq \mathscr{C}[z]$. Thus,
  $\mathscr{C}_\varepsilon[Z]$ is discrete in $Z$ and since $Z$ is
  collectionwise normal, there exists a discrete collection
  $\{V_C:C\in \mathscr{C}_\varepsilon[Z]\}$ of open subsets of $Z$
  such that ${C\subset V_C\subset \Delta(C,\varepsilon)}$,
  ${C\in \mathscr{C}_\varepsilon[Z]}$. Next, for each
  $C\in \mathscr{C}_\varepsilon[Z]$ take a clopen subset
  $U_C\subset Z$ with $C\subset U_C\subset V_C$, and set
  $\mathscr{U}_0=\{U_C: C\in
  \mathscr{C}_\varepsilon[Z]\}$. Accordingly, by
  \eqref{eq:sel-top-29:2} and Proposition
  \ref{proposition-sel-top-27:1},
  $\mathscr{U}_0\subset \mathscr{D}[X,\varepsilon]$. To construct the
  remaining part of the required partition $\mathscr{U}$, let
  $Y=Z\setminus \bigcup\mathscr{U}_0$ which is a clopen set with
  ${\diam(C)<\varepsilon}$, for every $C\in \mathscr{C}[Y]$. Hence,
  every $C\in \mathscr{C}[Y]$ is contained in an open set
  $W_C\subset Y$ such that $\diam(W_C)<\varepsilon$. So,
  $\{W_C:C\in \mathscr{C}[Y]\}$ is an open cover of $Y$ and each
  component of $Y$ is contained in some element of this
  cover. According to Corollary \ref{corollary-sel-top-14:3}, $Y$ has
  a discrete cover $\mathscr{V}$ which refines
  $\{W_C:C\in \mathscr{C}[Y]\}$. Then $Y$ also has a discrete cover
  $\mathscr{U}_1\subset \mathscr{D}[X]$ which refines
  $\{W_C:C\in \mathscr{C}[Y]\}$. Indeed, for each $V\in \mathscr{V}$,
  let $\mathscr{U}_V=\mathscr{C}[V]$ if $V$ contains finitely many
  components of $Z$, and $\mathscr{U}_V=\{V\}$ otherwise. Then
  $\mathscr{U}_1=\bigcup_{V\in \mathscr{V}}\mathscr{U}_V\subset
  \mathscr{D}[X]$ and by \eqref{eq:sel-top-29:2}, we also have that
  $\mathscr{U}_1\subset \mathscr{D}[X,\varepsilon]$. Thus,
  $\mathscr{U}=\mathscr{U}_0\cup \mathscr{U}_1\subset
  \mathscr{D}[X,\varepsilon]$ is a partition of $Z$. Finally, let us
  observe that we may always assume that $|\mathscr{U}|\neq 2$.
  Namely, in this case, $Z\notin \mathscr{D}[X,\varepsilon]$ and,
  therefore, $Z$ is not a component because $\dim(Z)\geq
  \varepsilon$. Hence, $Z$ contains infinitely many components because
  $Z\in \mathscr{D}[X]$, see \eqref{eq:sel-top-27:1}. If $\mathscr{U}$
  is finite, then an element $U\in \mathscr{U}$ also contains
  infinitely many components and by Proposition
  \ref{proposition-sel-top-27:1}, $U$ can be partitioned into two
  elements $U_1,U_2\in \mathscr{D}[X]$. Evidently,
  $U_1,U_2\in \mathscr{D}[X,\varepsilon]$ because
  $U\in \mathscr{D}[X,\varepsilon]$, and we may replace $U$ with $U_1$
  and $U_2$.
\end{proof}

\begin{lemma}
  \label{lemma-sel-top-15:1}
  Let $\varepsilon>0$ and $Z\in \mathscr{D}[X,2\varepsilon]$. If
  $C\in \mathscr{C}_{2\varepsilon}[Z]$ and $\ell(C),r(C)\in \{0,1\}$
  are as in \eqref{eq:sel-top-22:1}, then $Z$ has a discrete partition
  $\mathscr{U}\subset \mathscr{D}[X,\varepsilon]$ such that
  $|\mathscr{U}|\neq 2$ and
  \begin{enumerate}[label=\upshape{(\roman*)}]
  \item\label{item:sel-top-35:1} $U< C$ or $C< U$ for every $U\in
    \mathscr{U}$ with $U\cap C=\emptyset$,\smallskip
  \item\label{item:sel-top-35:2}
    $\big|\{U\in \mathscr{U}:U< C\}\big|\geq 6\ell(C)$ and
    $\big|\{U\in \mathscr{U}:C< U\}\big|\geq 6r(C)$.
  \end{enumerate}
  Moreover, $\diam(U)\leq 2\varepsilon$ for every $U\in \mathscr{U}$
  with $U\cap C=\emptyset$. 
\end{lemma}

\begin{proof}
  The set $\{S\in \mathscr{C}[Z]:S< C\}$ is infinite provided
  $\ell(C)=1$; similarly, so is the set $\{S\in \mathscr{C}[Z]:C< S\}$
  if $r(C)=1$. Hence, according to Corollary
  \ref{corollary-sel-top-10:2} and Proposition
  \ref{proposition-sel-top-27:1}, there is a clopen set
  $U_C\subset Z$, with $C\subset U_C\subset \Delta(C,\varepsilon)$,
  such that
  \begin{enumerate}
  \item \label{item:sel-top-35:3}
    $\big|\{S\in \mathscr{C}[Z\setminus U_C]: S< C\}\big|\geq
    6\ell(C)$\quad and\smallskip
  \item \label{item:sel-top-35:4}
    $\big|\{S\in \mathscr{C}[Z\setminus U_C]: C< S\}\big|\geq 6r(C)$.
  \end{enumerate}
  Let $p\in C$.  Then
  $Y_\ell=Z\cap \left[(\leftarrow,p)_{\leq}\setminus U_C\right]$ is a
  clopen subset of $Z$ such that ${S< C}$ for every component
  $S\in \mathscr{C}[Y_\ell]$. Moreover, by \ref{item:sel-top-35:3},
  $Y_\ell$ has at least $6\ell(C)$-many components. If
  $Y_\ell=\emptyset$, set $\mathscr{U}_1=\emptyset$. If $Y_\ell$ has
  finitely many components, take
  $\mathscr{U}_1=\mathscr{C}[Y_\ell]\subset \mathscr{D}[X]$. Finally,
  if $Y_\ell$ has infinitely many components, by Lemma
  \ref{lemma-sel-top-14:1}, it has a partition
  $\mathscr{U}_1\subset \mathscr{D}[X,\varepsilon]$ with
  $|\mathscr{U}_1|\neq 2$. In this case, precisely as in the proof of
  Lemma \ref{lemma-sel-top-14:1} (using Proposition
  \ref{proposition-sel-top-27:1}), we may assume that
  $|\mathscr{U}_1|\geq 6\geq 6\ell(C)$. Evidently, by
  \ref{item:sel-top-35:3}, $U< C$ for every $U\in
  \mathscr{U}_1$. Repeating the same argument but now with
  $Y_r=Z\cap \left[(p\to)_\leq\setminus U_C\right]$ instead of
  $Y_\ell$ and \ref{item:sel-top-35:4} instated of
  \ref{item:sel-top-35:3}, we get a partition
  $\mathscr{U}_2\subset \mathscr{D}[X,\varepsilon]$ of $Y_r$ such that
  $|\mathscr{U}_2|\geq 6r(C)$ and $C< U$ for every
  $U\in \mathscr{U}_2$. Then
  $\mathscr{U}=\mathscr{U}_1\cup\{U_C\}\cup \mathscr{U}_2$ is as
  required because $U_C\in \mathscr{D}[X,\varepsilon]$, see
  \eqref{eq:sel-top-29:2}.\smallskip

  Finally, let us see that $\diam(U)\leq 2\varepsilon$ for every
  $U\in \mathscr{U}$ with $U\cap C=\emptyset$. Indeed, for
  $U\in \mathscr{U}$, we have that $U\in \mathscr{D}[X,\varepsilon]$
  and, therefore, $\diam(U)\geq \varepsilon$ implies that
  $U\subset \Delta(S,\varepsilon)$ for some
  $S\in \mathscr{C}_\varepsilon[U]$. If moreover $U\cap C=\emptyset$,
  then $S\neq C$ and by Proposition \ref{proposition-sel-top-14:2},
  $\diam(S)\leq \varepsilon$ because
  $Z\in \mathscr{D}[X,2\varepsilon]$ and
  $C\in \mathscr{C}_{2\varepsilon}[Z]$, see \eqref{eq:sel-top-29:1}
  and \eqref{eq:sel-top-29:2}. Since
  $U\subset \Delta(S,\varepsilon)\subset
  \mathbf{O}\left(S,\frac\varepsilon2\right)$, see
  \eqref{eq:sel-top-23:1}, we finally get that
  $\diam(U)\leq 2 \varepsilon$.
\end{proof}

We conclude this section by extending the construction in Lemmas
\ref{lemma-sel-top-14:1} and \ref{lemma-sel-top-15:1} to a system of
discrete covers on $X$. To this end, let us recall that a partially
ordered set $(T,\preceq)$ is a \emph{tree} if $\{s\in T:s\preceq t\}$
is well ordered, for every $t\in T$. For a tree $(T,\preceq)$, we use
${T}(0)$ to denote the minimal elements of $T$.  Given an ordinal
$\alpha$, if ${T}(\beta)$ is defined for every $\beta<\alpha$, then
${T}(\alpha)$ denotes the minimal elements of
$T\setminus \bigcup\{{T}(\beta):\beta< \alpha\}$.  The set
${T}(\alpha)$ is called the \emph{$\alpha^{\text{th}}$-level} of $T$,
while the \emph{height} of $T$ is the least ordinal $\alpha$ such that
$T=\bigcup\{{T}(\beta):\beta< \alpha\}$. We say that $T$ is
\emph{$\alpha$-levelled} if its height is $\alpha$. A maximal linearly
ordered subset of $T$ is called a \emph{branch}, and $\mathscr{B}(T)$
is used to denote the set of all branches of $T$.\medskip

For a tree $(T,\preceq)$, the \emph{node} of $t\in T$ is the subset
$\node(t)\subset T$ of all immediate successors of $t$, and we say
that $T$ is \emph{pruned} if $\node(t)\neq \emptyset$, for every
$t\in T$. In these terms, an $\omega$-levelled tree $(T,\preceq)$ is
pruned if each branch $\beta\in \mathscr{B}(T)$ is infinite.  In what
follows, we will write $\mathscr{S}:T\sto X$ to designate that
$\mathscr{S}$ is a \emph{set-valued} (or \emph{multi-valued})
\emph{mapping} from $T$ to the nonempty subsets of $X$.  In these
terms, for a pruned $\omega$-levelled tree $(T,\preceq)$, a mapping
$\mathscr{S}:T\sto X$ is a \emph{sieve} on $X$ if
$X=\bigcup\{\mathscr{S}(t):t\in T(0)\}$ and
$\mathscr{S}(t)=\bigcup\{\mathscr{S}(s):s\in \node(t)\}$ for every
$t\in T$. A sieve $\mathscr{S}:T\sto X$ on $X$ is called
\emph{discrete} if each indexed family
$\{\mathscr{S}(t): t\in T(n)\}$, $n<\omega$, is a discrete cover of
$X$. \medskip

In the theorem below, an important role will be played by special
trees. Namely, we shall say that a tree $(T,\preceq)$ is
\emph{anti-binary} if $|T(0)|\neq 2$ and $|\node(t)|\neq 2$, for every
$t\in T$.  Moreover, if $C\in \mathscr{C}_*[X]$,
then we will use $\kappa[C]$ to denote the least $n<\omega$ such that
${C\in \mathscr{C}_{2^{-n}}[X]}$, i.e.\
\begin{equation}
  \label{eq:sel-top-34:2}
  \kappa[C]=\min\big\{n<\omega: \diam(C)\geq 2^{-n}\big\}.
\end{equation}

\begin{theorem}
  \label{theorem-sel-top-23:1}
  There exists an anti-binary pruned $\omega$-levelled tree $T$ and a
  discrete sieve $\mathscr{S}:T\to \mathscr{D}[X]$ on $X$ such that
  $\mathscr{S}\uhr T(n): T(n)\to \mathscr{D}[X,2^{-n}]$, for every
  $n<\omega$.  Moreover, if $C\in \mathscr{C}_*[X]$ and
  $\beta\in \mathscr{B}(T)$ is the branch with
  $C\subset \bigcap_{t\in \beta}\mathscr{S}(t)$, then for
  $n\geq \kappa[C]$, $t\in \beta\cap T(n)$ and
  $p\in \beta\cap \node(t)$, 
  \begin{enumerate}[label=\upshape{(\thesection.\arabic*)}]
    \addtocounter{enumi}{6}  
  \item\label{item:3}
    $\begin{cases}
      \diam(\mathscr{S}(s))\leq 2^{-n+1} &\text{and}\\
      \mathscr{S}(s)< C\ \text{or}\ C< \mathscr{S}(s),
    &\text{whenever $s\in \node(t)$ with
      $s\neq p$,}
    \end{cases}$\smallskip
  \item\label{item:4}
    $
    \begin{cases}
      \big|\{s\in \node(t):\mathscr{S}(s)< C\}\big|\geq 6\ell(C)\quad
      \text{and}\\
      {\big|\{s\in \node(t):C< \mathscr{S}(s)\}\big|\geq 6r(C)}.
    \end{cases}
    $ 
  \end{enumerate}
\end{theorem}

\begin{proof}
  By Lemma \ref{lemma-sel-top-14:1}, $X$ has a discrete partition
  $\mathscr{U}\subset \mathscr{D}[X,2^0]$ such that
  $|\mathscr{U}|\neq 2$. Take $T(0)=\mathscr{U}$ and let
  $\mathscr{S}:T(0)\to \mathscr{U}$ be the identity of $\mathscr{U}$.
  To construct the next level $T(1)$ of the tree and the values
  $\mathscr{S}(s)$, $s\in T(1)$, of the sieve $\mathscr{S}$, it
  suffices to construct discrete partitions
  $\mathscr{U}_t\subset \mathscr{D}\left[X,2^{-1}\right]$ of each
  $\mathscr{S}(t)$, $t\in T(0)$, such that $|\mathscr{U}_t|\neq 2$ and
  \ref{item:3} and \ref{item:4} hold for the elements of
  $\mathscr{U}_t$. Then we can set $\node(t)=(\mathscr{U}_t,t)$,
  $t\in T(0)$, and let $\mathscr{S}\uhr \node(t)$ be the projection
  on the first factor. Finally, we may take
  $T(1)=\bigcup_{t\in T(0)}\node(t)$ and define $t\prec s$ whenever
  $t\in T(0)$ and $s\in\node(t)$.\smallskip

  Turning to the construction of the partitions $\mathscr{U}_t$ of
  $\mathscr{S}(t)$, $t\in T(0)$, we distinguish the following two
  cases.  If $\mathscr{C}_{2^0}[\mathscr{S}(t)]=\emptyset$, then one
  can take $\mathscr{U}_t\subset \mathscr{D}\left[X,2^{-1}\right]$ as
  in Lemma \ref{lemma-sel-top-14:1} applied with
  $Z=\mathscr{S}(t)$. Otherwise, if
  $C\in \mathscr{C}_{2^0}[\mathscr{S}(t)]$, then by
  \eqref{eq:sel-top-29:1} and \eqref{eq:sel-top-29:2},
  $\mathscr{S}(t)\subset \Delta(C,2^{0})$ and it follows from
  Proposition \ref{proposition-sel-top-14:2} that
  $\mathscr{C}_{2^0}[\mathscr{S}(t)]=\{C\}$ is a singleton.  Hence, in
  this case, we may take $\mathscr{U}_t\subset \mathscr{D}[X,2^{-1}]$
  as in Lemma \ref{lemma-sel-top-15:1} applied with $Z=\mathscr{S}(t)$
  and the component $C\in
  \mathscr{C}_{2^0}[\mathscr{S}(t)]$. According to
  \ref{item:sel-top-35:1} and \ref{item:sel-top-35:2} of Lemma
  \ref{lemma-sel-top-15:1}, the resulting values on the associated
  sieve $\mathscr{S}$ on $\node(t)=(\mathscr{U}_t,t)$, i.e.\ the
  projection on $\mathscr{U}_t$, will satisfy \ref{item:3} and
  \ref{item:4} with respect to this component $C$ and the element
  $p\in \node(t)$ with $\mathscr{S}(p)=C$. The construction can be
  carried on by induction.
\end{proof}
  
\section{Sieve-Invariant Weak Selections} 

In this section, we finalise the proof of the following theorem
dealing with the cws-number of metrizable suborderable spaces.

\begin{theorem}
  \label{theorem-sel-top-19:1}
  If $X$ is a suborderable metrizable space which has infinitely many
  components, then $\cws(X)=1$.
\end{theorem}

The proof of Theorem \ref{theorem-sel-top-19:1} is based on two
concepts associated to sieves.  The one is about invariant weak
selections. Namely, for a sieve $\mathscr{S}:T\sto X$ on a set $X$, we
shall say that a weak selection $g\in \sel_2(X)$ is
\emph{$\mathscr{S}$-invariant} if each of the families
$\{\mathscr{S}(t):t\in T(n)\}$, $n<\omega$, is $g$-decisive. Such a
sieve $\mathscr{S}$ is not  as arbitrary as it might seem at
first. Indeed, in this case each family $\{\mathscr{S}(t):t\in
T(n)\}$, $n<\omega$, must be pairwise disjoint being $g$-decisive.  In
particular, for each point $x\in X$ there exists a unique branch
$\beta[x]\in \mathscr{B}(T)$ with $x\in \bigcap_{t\in \beta[x]}
\mathscr{S}(t)$. So, it is defined a natural map $\beta:X\to
\mathscr{B}(T)$, namely
\begin{equation}
  \label{eq:sel-top-1st:1}
  X\ni x\quad\longrightarrow\quad \beta[x]=\{t\in T: x\in
  \mathscr{S}(t)\}\in 
  \mathscr{B}(T). 
\end{equation}
This map is useful to handle the continuity of $\mathscr{S}$-invariant
weak selections.

\begin{proposition}
  \label{proposition-sel-top-35:1}
  Let $\mathscr{S}:T\sto X$ be a discrete sieve on a space $X$ and $g$
  be an $\mathscr{S}$-invariant weak selection for $X$. Then $g$ is
  continuous at each pair $\{x,y\}\in \mathscr{F}_2(X)$ with
  $\beta[x]\neq \beta[y]$.
\end{proposition}

\begin{proof}
  Take points $x,y\in X$ with $x<_g y$ and $\beta[x]\neq
  \beta[y]$. Then there exists $n<\omega$ and elements
  $s\in \beta[x]\cap T(n)$ and $t\in \beta[y]\cap T(n)$ such that
  $s\neq t$. Accordingly, $x\in \mathscr{S}(s)$, $y\in \mathscr{S}(t)$
  and $\mathscr{S}(s)<_g \mathscr{S}(t)$ because $g$ is
  $\mathscr{S}$-invariant and $x<_g y$. Since $\mathscr{S}$ is
  discrete, $\mathscr{S}(s)$ and $\mathscr{S}(t)$ are open
  sets. Therefore, $g$ satisfies the continuity condition at
  $\{x,y\}\in \mathscr{F}_2(X)$, see \cite[Theorem
  3.1]{gutev-nogura:01a}.
\end{proof}

Let $\mathscr{S}:T\sto X$ be a discrete sieve on $X$ such that
$\beta[x]\neq \beta[y]$, whenever $x,y\in X$ with $x\neq y$, i.e.\ the
sieve $\mathscr{S}$ is \emph{separating the points of $X$}. Then
according to Proposition \ref{proposition-sel-top-35:1}, each
$\mathscr{S}$-invariant weak selection $g\in\sel_2(X)$ is
continuous. This situation is however very restrictive being
applicable only for totally disconnected spaces. Indeed, if $X$ has a
non-degenerate component, then it cannot have a sieve which separates
its points. For such spaces, the best that can be achieved is to
separate the components. Namely, we shall say that a sieve
$\mathscr{S}:T\sto X$ is \emph{separating the components of $X$} if
for every two different components $P,Q\in \mathscr{C}[X]$ there are
different elements $s,t\in T$ with $P\subset \mathscr{S}(s)$ and
$Q\subset \mathscr{S}(t)$. In terms of the map
$\beta:X\to \mathscr{B}(T)$, see \eqref{eq:sel-top-1st:1}, this can be
expressed by the property that
$\mathscr{C}[x]=\bigcap_{t\in \beta[x]}\mathscr{S}(t)$, whenever
$x\in X$. The sieve constructed in Theorem \ref{theorem-sel-top-23:1}
has this property.

\begin{lemma}
  \label{lemma-sel-top-34:1}
  Let $(X,\leq)$ be a suborderable metrizable space which has
  infinitely many components, $\rho$ be a convex metric on $X$
  compatible with the topology of $X$, and
  $\mathscr{S}:T\to \mathscr{D}[X]$ be a discrete sieve on $X$ as in
  Theorem \ref{theorem-sel-top-23:1}. Then for every clopen set
  $U\subset X$ and $x\in U$, there exists $t\in \beta[x]$ with
  $\mathscr{S}(t)\subset U$.
\end{lemma}

\begin{proof}
  Let $U\subset X$ be a clopen set, $x\in U$ and
  $\beta[x]=\{t_n\in T(n): n<\omega\}\in \mathscr{B}(T)$ be the
  associated branch as in \eqref{eq:sel-top-1st:1}.  If
  $\mathscr{C}[x]\in \mathscr{C}_*[X]$, then
  $\mathscr{C}[x]\in \mathscr{C}_{2^{-\kappa}}[X]$ for some
  $\kappa<\omega$. Since
  $\mathscr{C}[x]\subset \mathscr{S}(t_n)\in \mathscr{D}[X,2^{-n}]$,
  $n\geq \kappa$, it follows from \eqref{eq:sel-top-23:1} and
  \eqref{eq:sel-top-29:2} that
  $\mathscr{S}(t_n)\subset \Delta(\mathscr{C}[x],2^{-n})\subset
  \mathbf{O}(\mathscr{C}[x], 2^{-n})$, for every
  $n\geq\kappa$. Therefore, by Proposition
  \ref{proposition-sel-top-34:1}, $\mathscr{S}(t_n)\subset U$ for some
  $n\geq \kappa$, because $\mathscr{C}[x]\subset U$ and $U$ is
  open.\smallskip

  If $\mathscr{C}[x]$ is a singleton, then $\mathscr{C}[x]=\{x\}$. In
  this case, it suffices to show that for every $\kappa<\omega$ there
  exists $m\geq \kappa$ such that
  $\diam(\mathscr{S}(t_m))\leq 2^{-\kappa+1}$. So, take any
  $\kappa<\omega$. If $\diam(\mathscr{S}(t_\kappa))\geq 2^{-\kappa}$,
  then $\mathscr{S}(t_\kappa)\subset \Delta(C,2^{-\kappa})$ for some
  $C\in \mathscr{C}_{2^{-\kappa}}[\mathscr{S}(t_\kappa)]$ because
  $\mathscr{S}(t_\kappa)\in \mathscr{D}[X,2^{-\kappa}]$, see
  \eqref{eq:sel-top-29:2}.  For this $C$, for the same reason, we have
  that
  $\mathscr{S}(t)\subset \Delta(C,2^{-n})\subset \mathbf{O}(C,2^{-n})$
  for every $n\geq \kappa$ and $t\in T(n)$ with
  $C\subset \mathscr{S}(t)$. However, $C\neq \mathscr{C}[x]=\{x\}$
  and, therefore, $x\notin \mathbf{O}(C,2^{-n})$ for some $n>
  \kappa$. In other words, $\mathscr{S}(t_n)\cap C=\emptyset$ for some
  $n>\kappa$ and we may take the maximal $n\geq \kappa$ for which
  $C\subset \mathscr{S}(t_n)$. Then $t_{n+1}\in \node(t_n)$ and
  $\mathscr{S}(t_{n+1})\cap C=\emptyset$. Hence, by \ref{item:3} of
  Theorem \ref{theorem-sel-top-23:1},
  $\diam(\mathscr{S}(t_{n+1}))\leq 2^{-n+1}\leq 2^{-\kappa+1}$.
\end{proof}

Finally, we are also ready for the proof of Theorem
\ref{theorem-sel-top-19:1}.

\begin{proof}[Proof of Theorem \ref{theorem-sel-top-19:1}]
  Let $\leq$ be a compatible linear order on $X$, and $\rho$ be an
  admissible convex metric on $X$. Also, let
  $\mathscr{S}:T\to \mathscr{D}[X]$ be a discrete sieve on $X$ with
  the properties in Theorem \ref{theorem-sel-top-23:1}. We are going
  to construct an $\mathscr{S}$-invariant continuous weak selection
  $g$ for $X$ such that $\mathscr{T}_g$ is with the topology
  $\mathscr{T}$ of $X$. To this end, for convenience, set
  $\node(\emptyset)=T(0)$ and $T_\emptyset=T\cup\{\emptyset\}$ so that
  we may represent the tree partitioned by its nodes, namely
  $T=\bigcup_{t\in T_\emptyset}\node(t)$. Then, a weak selection
  $g\in \sel_2(X)$ is $\mathscr{S}$-invariant if and only if each
  family $\{\mathscr{S}(s):s\in \node(t)\}$, $t\in T_\emptyset$, is
  $g$-decisive. This interpretation allows to construct
  sieve-invariant weak selections by using an inductive
  argument.\smallskip 
  
  Let $\mathscr{S}(\emptyset)=X\in \mathscr{D}[X]$ so that we may
  consider the sieve $\mathscr{S}$ as a mapping
  $\mathscr{S}:T_\emptyset\to \mathscr{D}[X]$. If
  $\mathscr{C}_{2^{-n}}[\mathscr{S}(t)]\neq \emptyset$ for some
  $t\in T(n)$, then $\mathscr{C}_{2^{-n}}[\mathscr{S}(t)]$ is a
  singleton because
  $\mathscr{S}\uhr T(n):T(n)\to \mathscr{D}\left[X,2^{-n}\right]$, see
  \eqref{eq:sel-top-29:2} and Proposition
  \ref{proposition-sel-top-14:2}. In this case, let $C[t]$ be the
  unique element of $\mathscr{C}_{2^{-n}}[\mathscr{S}(t)]$ and
  $\pi(t)\in \node(t)$ the element with
  $C[t]\subset \mathscr{S}(\pi(t))$. Otherwise, if
  $\mathscr{C}_{2^{-n}}[\mathscr{S}(t)]=\emptyset$, for technical
  reasons only, set $C[t]=\emptyset$. In these terms, for each
  $t\in T_\emptyset$ we will construct a weak selection $\sigma_t$ for
  $\{\mathscr{S}(s): s\in\node(t)\}$ such that for every $s,s^*\in
  \node(t)$, 
  \begin{enumerate}[label=\upshape{(\thesection.\arabic*)}]
  \addtocounter{enumi}{1}
  \item\label{item:9} $\mathscr{S}(s)\in \mathscr{T}_{\sigma_t* h_t}$,
    whenever $h_t\in\sel_2(\mathscr{S}(t))$;\smallskip
  \item\label{item:10} If $C[t]\neq\emptyset$ and $s\neq \pi(t)\neq
    s^*$, then\smallskip 
    \begin{enumerate}[label=(5.3\alph*)]
    \item\label{item:sel-top-36:1}
      $\mathscr{S}(s)<_{\sigma_t} \mathscr{S}(\pi(t))$ provided
      $\mathscr{S}(s)< C[t]$ and
      $\mathscr{S}(\pi(t))<_{\sigma_t} \mathscr{S}(s^*)$ provided
      $C[t]<\mathscr{S}(s^*)$; 
    \item\label{item:sel-top-36:2}
      $\mathscr{S}(s)<_{\sigma_t} \mathscr{S}(s^*)$ whenever
      $\mathscr{S}(s)<C[t]<\mathscr{S}(s^*)$. 
    \end{enumerate}
  \end{enumerate}
  \addtocounter{equation}{2}

  According to Lemma \ref{lemma-sel-top-4:1},
  $\{\mathscr{S}(t): t\in \node(\emptyset)\}$ has a weak selection
  $\sigma_\emptyset$ with
  $\mathscr{S}(t)\in \mathscr{T}_{\sigma_\emptyset* h}$, for every
  $t\in\node(\emptyset)$ and $h\in \sel_2(X)$. The construction can be
  carried on by induction. Namely, take $t\in\node(\emptyset)$. If
  $C[t]=\emptyset$, then for the same reason,
  $\{\mathscr{S}(s): s\in \node(t)\}$ has a weak selection $\sigma_t$
  such that \ref{item:9} holds for this particular $t$. Suppose that
  $C[t]\neq\emptyset$, in which case
  $\mathscr{S}(t)\subset \Delta(C[t],2^0)$ because
  $\mathscr{S}(t)\in \mathscr{D}[X,2^0]$, see \eqref{eq:sel-top-29:2}
  and Theorem \ref{theorem-sel-top-23:1}.  If
  ${\ell(C[t])=0= r(C[t])}$, it follows from \eqref{eq:sel-top-23:1}
  that $\mathscr{S}(t)=C[t]$. Therefore, $\node(t)=\{\pi(t)\}$ and
  $\mathscr{S}(\pi(t))=\mathscr{S}(t)$, so \ref{item:9} and
  \ref{item:10} hold in a trivial way. If $\ell(C[t])=1\neq r(C[t])$,
  then $C< C[t]$ for every component
  $C\in \mathscr{C}[\mathscr{S}(t)]$ with $C\neq C[t]$. Hence, by
  \ref{item:3} and \ref{item:4} of Theorem \ref{theorem-sel-top-23:1},
  $\big|\node(t)\setminus\{\pi(t)\}\big|\geq 6$ and
  $\mathscr{S}(s)< C[t]$ for each $s\in
  \node(t)\setminus\{\pi(t)\}$. Now, we may apply Lemma
  \ref{lemma-sel-top-32:1} to get a weak selection $\sigma_t$ for
  $\{\mathscr{S}(s):s\in \node(t)\}$ such that \ref{item:9} holds and
  $\mathscr{S}(s)<_{\sigma_t} \mathscr{S}(\pi(t))$ for every
  $s\in \node(t)\setminus\{\pi(t)\}$. Accordingly, \ref{item:10} holds
  as well. The case when $\ell(C[t])=0\neq r(C[t])$ can be handled in
  exactly the same way using Lemma \ref{lemma-sel-top-32:1} and
  reversing the selection relation obtained in that lemma. Finally,
  suppose that $\ell(C[t])=1=r(C[t])$. In this case, by \ref{item:3}
  of Theorem \ref{theorem-sel-top-23:1}, $\mathscr{S}(s)< C[t]$ or
  $C[t]< \mathscr{S}(s)$ for every $s\in\node(t)\setminus \{\pi(t)\}$,
  while by \ref{item:4} of the same theorem,
  \[
    \begin{cases}
      \big|\{s\in \node(t):\mathscr{S}(s)< C[t]\}\big|\geq 6\quad
      \text{and}\\
      {\big|\{s\in \node(t):C[t]< \mathscr{S}(s)\}\big|\geq 6}.
    \end{cases}
  \]
  The weak selection $\sigma_t$ for $\{\mathscr{S}(s):s\in \node(t)\}$
  defined as in Corollary \ref{corollary-sel-top-14:1} is now as
  required in \ref{item:9} and \ref{item:10}.\smallskip

  Having already constructed the selections $\sigma_t$,
  $t\in T_\emptyset$, we finalise the proof by showing that they
  generate a continuous $\mathscr{S}$-invariant weak selection $g$ for
  $X$ with $\mathscr{T}_g=\mathscr{T}$.  Namely, let $x,y\in X$. If
  $\mathscr{C}[x]=\mathscr{C}[y]$, then set $x\leq_g y$ provided that
  $x\leq y$. If $\mathscr{C}[x]\neq \mathscr{C}[y]$, then by Lemma
  \ref{lemma-sel-top-34:1}, there are unique elements
  $t[x,y]\in T_\emptyset$ and $s[x],s[y]\in \node(t[x,y])$ such that
  $s[x]\neq s[y]$ and
  \begin{equation}
    \label{eq:sel-top-29:3}
    \mathscr{C}[x],\mathscr{C}[y]\subset \mathscr{S}(t[x,y]),\
    \mathscr{C}[x]\subset \mathscr{S}(s[x])\  \text{and}\
    \mathscr{C}[y]\subset \mathscr{S}(s[y]).
  \end{equation}
  In this case, set $x<_g y$ provided that
  $\mathscr{S}(s[x])<_{\sigma_{t[x,y]}} \mathscr{S}(s[y])$. Thus, we
  get an $\mathscr{S}$-invariant selection $g\in\sel_2(X)$ such that
  by \ref{item:9}, $\mathscr{S}(t)\in \mathscr{T}_g$ for every
  $t\in T$.\smallskip

  To show that $g$ is continuous, we have to show that it is
  continuous at each pair of points $x,y\in X$ with $x\neq y$. If
  $\mathscr{C}[x]\neq \mathscr{C}[y]$, then this follows from
  Proposition \ref{proposition-sel-top-35:1} because by Lemma
  \ref{lemma-sel-top-34:1}, $\beta[x]\neq \beta[y]$. Hence, the
  verification of continuity of $g$ is reduced to the case when
  $\mathscr{C}[x]=\mathscr{C}[y]$. Below, we will show that this is
  also related to the verification of $\mathscr{T}_g=\mathscr{T}$. To
  this end, let us show that for every non-degenerated component
  $C\in \mathscr{C}[X]$ there exists $t\in T$ with
  $C\subset \mathscr{S}(t)$, such that for every cut point $p\in C$,
  \begin{equation}
    \label{eq:sel-top-34:4}
    \begin{cases}
      (\gets,p)_{\leq}\cap \mathscr{S}(t), (p,\to)_{\leq}\cap
      \mathscr{S}(t)\in \mathscr{T}_g\quad \text{and}\\
      (\gets,p)_{\leq}\cap \mathscr{S}(t)<_g (p,\to)_{\leq}\cap
      \mathscr{S}(t).
    \end{cases}
  \end{equation}

  Let $\kappa=\kappa[C]$ be as in \eqref{eq:sel-top-34:2}, i.e.\ the
  minimal $k< \omega$ with $C\in \mathscr{C}_{2^{-k}}[X]$, see
  \eqref{eq:sel-top-29:1}. Also, let
  $\beta=\{t_n\in T(n): n<\omega\}\in \mathscr{B}(T)$ be the unique
  branch with $C\subset\bigcap_{n<\omega}\mathscr{S}(t_n)$. We will
  show that $\mathscr{S}(t_\kappa)$ is as required in
  \eqref{eq:sel-top-34:4}. So, take a cut point $p\in C$ and
  $x,y\in \mathscr{S}(t_\kappa)$ with $x<p<y$. If $x,y\in C$, then
  $x<_g y$ because on each component, the selection relation $\leq_g$
  is the linear order $\leq$ on $X$. Otherwise, if
  $\mathscr{C}[x]\neq \mathscr{C}[y]$, set
  $n=\max\{k<\omega:\{x,y\}\subset \mathscr{S}(t_k)\}$. Also, let
  $t[x,y],s[x],s[y]\in T$ be as in \eqref{eq:sel-top-29:3}. We now
  have that $n\geq \kappa$ and $t_n=t[x,y]$, so
  $s[x],s[y]\in \node(t_n)$, and the property follows from
  \ref{item:10}. Indeed, if $s[y]=t_{n+1}$, i.e.\
  $y\in \mathscr{S}(t_{n+1})$, then $\mathscr{S}(s[x])< C$ and by
  \ref{item:sel-top-36:1} and \eqref{eq:sel-top-29:3}, $x<_g
  y$. Similarly, $x<_g y$ provided $s[x]=t_{n+1}$. Finally, if
  $s[x]\neq t_{n+1}\neq s[y]$, this follows from
  \ref{item:sel-top-36:2}. Thus, the second part of
  \eqref{eq:sel-top-34:4} holds. To show the first part of this
  property, we essentially repeat the same argument. Namely, to show
  that $(p,\to)_{\leq}\cap \mathscr{S}(t_\kappa)\in \mathscr{T}_g$, it
  suffices to show that
  $(p,\to)_{\leq}\cap \mathscr{S}(t_\kappa)= (p,\to)_{\leq_g}\cap
  \mathscr{S}(t_\kappa)$. So, take a point
  $q\in \mathscr{S}(t_\kappa)$. If $q\in C$, then $p<_g q$ precisely
  when $p< q$. Suppose that $q\notin C$, and let
  $t[p,q], s[p],s[q]\in T$ be as in \eqref{eq:sel-top-29:3}. Then
  $t[p,q]=t_n$ for some $n\geq \kappa$, while
  $s[p]=t_{n+1}\in \node(t_n)$ and $s=s[q]\in \node(t_n)$. According
  to \ref{item:3} of Theorem \ref{theorem-sel-top-23:1},
  $\mathscr{S}(s)<C$ or $C<\mathscr{S}(s)$. Hence, by
  \ref{item:sel-top-36:1}, we have that
  $\mathscr{S}(t_{n+1})<_{\sigma_{t_n}} \mathscr{S}(s)$ precisely when
  $p<q$ because $p\in C$ and $q\in \mathscr{S}(s)$. Finally, by
  \eqref{eq:sel-top-29:3}, this implies that $p<_g q$ precisely when
  $p<q$. This completes the verification of \eqref{eq:sel-top-34:4}.
  \smallskip

  We are now ready to finalise the proof of Theorem
  \ref{theorem-sel-top-19:1}. Namely, let $x,y\in X$ with $x\neq
  y$. If $\mathscr{C}[x]\neq \mathscr{C}[y]$, as mentioned above, $g$
  is continuous at the pair $\{x,y\}\in \mathscr{F}_2(X)$. Suppose
  that $\mathscr{C}[x]=\mathscr{C}[y]$ and $x<y$. Next, take any point
  $p\in C=\mathscr{C}[x]$ with $x<p<y$. Finally, let $t\in T$ be as in
  \eqref{eq:sel-top-34:4} with respect to this component. Then
  $(\gets,p)_{\leq}\cap \mathscr{S}(t)<_g (p,\to)_{\leq}\cap
  \mathscr{S}(t)$ because $p$ is a cut point of $C$. Since
  $(\gets,p)_{\leq}\cap \mathscr{S}(t)$ and
  $(p,\to)_{\leq}\cap \mathscr{S}(t)$ are $\mathscr{T}$-open sets, $g$
  in continuous at $\{x,y\}\in \mathscr{F}_2(X)$ as well. Thus, we
  also get that $\mathscr{T}_g\subset \mathscr{T}$. As for the inverse
  inclusion, let us observe that by Lemma \ref{lemma-sel-top-34:1},
  each $\mathscr{T}$-clopen subset of $X$ is
  $\mathscr{T}_g$-open. This implies that
  $(\gets,x]_{\leq}\in \mathscr{T}_g$ whenever
  $(\gets,x]_{\leq}\in \mathscr{T}$ for some $x\in X$; similarly,
  $[x,\to)_{\leq}\in \mathscr{T}_g$ provided
  $[x,\to)_{\leq}\in \mathscr{T}$. Hence, to show that
  $\mathscr{T}\subset \mathscr{T}_g$, it remains to see that
  $(\gets,x)_{\leq},(x,\to)_{\leq}\in \mathscr{T}_g$ for every
  $x\in X$. So, take points $x,y\in X$ with $x<y$. If
  $\mathscr{C}[x]\neq \mathscr{C}[y]$, then
  $\mathscr{C}[y]\subset (x,\to)_{\leq}$ and by Corollary
  \ref{corollary-sel-top-10:2}, there exists a $\mathscr{T}$-clopen
  set $U\subset X$ with
  $\mathscr{C}[y]\subset U\subset (x,\to)_{\leq}$. Hence, by Lemma
  \ref{lemma-sel-top-34:1}, there also exists $t\in T$ with
  $\mathscr{C}[y]\subset \mathscr{S}(t)\subset U\subset
  (x,\to)_{\leq}$. Accordingly, $y$ is a $\mathscr{T}_g$-interior
  point of $(x,\to)_{\leq}$. Finally, suppose that
  $\mathscr{C}[x]=\mathscr{C}[y]$ and take a point
  $p\in C=\mathscr{C}[x]$ with $x<p<y$. Also, let $t\in T$ be as in
  \eqref{eq:sel-top-34:4} with respect to this component $C$. Since
  $p$ is a cut point of $C$, it follows from \eqref{eq:sel-top-34:4}
  that $y\in (p,\to)_{\leq}\cap \mathscr{S}(t)\in
  \mathscr{T}_g$. Since $(p,\to)_{\leq}\subset (x,\to)_{\leq}$, this
  implies again that $y$ is a $\mathscr{T}_g$-interior point of
  $(x,\to)_{\leq}$. Thus, $(x,\to)_{\leq}\in
  \mathscr{T}_g$. Similarly, $(\gets,x)_{\leq}\in \mathscr{T}_g$ and
  the proof is complete.
\end{proof}

\end{document}